\documentclass[reqno]{amsart}
\usepackage{hyperref}
\usepackage{mathrsfs}
\usepackage{mathtools}
\usepackage{pdfpages}

\begin{document}
\title[time-fractional g-Navier-Stokes equations]
{weak solutions to the time-fractional g-Navier-Stokes equations and optimal control}

\author[S. Ben Aadi]
{Sultana Ben Aadi}

\author[K. Akhlil]
{Khalid Akhlil}

\author[K. Aayadi]
{Khadija Aayadi}

\address{Plydisciplinary Faculty of Ouarzazate}

\address{sultana.benaadi@edu.uiz.ac.ma}
\address{k.akhlil@uiz.ac.ma}
\address{khadija.aayadi@gmail.com}

\date{\today}
\thanks{}
\subjclass[2000]{}

\begin{abstract}
In this paper we introduce the $g-$Navier-Stokes equations with time-fractional derivative of order $\alpha\in(0,1)$ in domains of $\mathbb R^2$. We then study the existence and uniqueness of weak solutions by means of Galerkin approximation. Finally, an optimal control problem is considered and solved.
\end{abstract}

\maketitle \setlength{\textheight}{19.5 cm}
\setlength{\textwidth}{12.5 cm}
\newtheorem{theorem}{Theorem}[section]
\newtheorem{lemma}[theorem]{Lemma}
\newtheorem{proposition}[theorem]{Proposition}
\newtheorem{corollary}[theorem]{Corollary}
\newtheorem{Hypo}[theorem]{Assumption}
\theoremstyle{definition}
\newtheorem{definition}[theorem]{Definition}
\newtheorem{assumptions}[theorem]{Assumptions}
\newtheorem{example}[theorem]{Example}
\theoremstyle{remark}
\newtheorem{remark}[theorem]{Remark}
\numberwithin{equation}{section} \setcounter{page}{1}

\newcommand{\mbbx}{\mathbb X}
\newcommand{\me}{\mathscr E}
\newcommand{\X}{\mathbb X}
\newcommand{\dfc}{\mathbb D_t^\alpha}

\section{Introduction}

Hale and Raugel \cite{HR92a, HR92b} studied 3d nonlinear equations in thin domains of the form  $ \Omega_\varepsilon=\Omega\times(0,\varepsilon)$,  where $\Omega\subset \mathbb R^2$ and $0<\varepsilon<1$. Afterwards Raugel and Sell \cite{RS93} studied Navier-Stokes equations in thin domains. This developments found applications in Lake equations and Shallow water equations \cite{CHL96,LOT96a,LOT96b}. J. Roh in his doctoral thesis \cite{R01}, has employed the techniques developed by Hale, Raugel and Sell to the Navier-Stokes equations on thin domains of the form $ \Omega_g=\Omega\times(0,g)$, where $g$ is some smooth scalar function. The derived equations are called the $g-$Navier-Stokes equations. The main difference between the classical Navier-Stokes equations and the $g-$Navier-Stokes equations is a weighted divergence condition of the form $\nabla.(gu)=0$. 

Much has already been done in the theory of $g-$Navier-Stokes equations. The existence of weak solutions is carried out in \cite{BR04}; see also \cite{R01, R05, QT17}. The stability and long-time behaviour questions can be found in \cite{AQ12a, AQ12b, Q14, ATT17}. Attractors of g-Navier-Stokes equations attracted most of the interest of researchers \cite{JHW11, JH09, JH10, JW13, K12, KR05, KKR06, Q15, Q16, R06, W09, W10}.

In this paper, we propose to study the weak solutions to time-fractional $g-$Navier-Stokes equations in the following form 

\begin{equation}\label{eq00}
 \left \{
   \begin{array}{r c l}
       \partial_t^\alpha u-\nu\Delta u+(u.\nabla)u+\nabla p & = & f\quad\text{ in } (0,T)\times\Omega \\
     \nabla.(gu) & = & 0 \quad\text{ in } (0,T)\times\Omega\\
     u(0,.)& = & u_0 \quad\text{ in }\Omega
   \end{array}
   \right .
\end{equation}
where where $\partial_t^\alpha$ is the Caputo fractional derivative of order $\alpha\in(0,1)$, $u = u(x, t) = (u_1, u_2)$ is the unknown velocity vector, $p = p(x, t)$ is the unknown pressure, $\nu>0$ is the kinematic viscosity coefficient, $u_0$ is the initial velocity, and $f$ represents the external force field.

The case where $\alpha=1$ and $g=1$ is just the standard Navier-Stokes equations and is a well-know subject, for more details, we refer to the monographs \cite{mah20,zaj98,Tac19,G11,L02,LK16} and references therein. The case where $\alpha=1$ and $g\neq 1$ is the $g-$Navier-Stokes equations as described above. The case where $0<\alpha<1$ and $g=1$ was considered first by Lions in \cite{L59} but for order less than $\frac{1}{4}$
provided the space dimension is not further than $4$.  In a recent work of Zhou and Peng \cite{ZP17}, the question of weak solutions and optimal control problem of time-fractional Navier-Stokes equations is considered. In the current paper we will consider the case where  $0<\alpha<1$ and $g\neq 1$. We will prove that problem \ref{eq00} has a unique weak solution in domains $\Omega\subset \mathbb R^2$ with enough smooth boundaries.

As a second step, we will consider the following control problem associated with equations \eqref{eq00}:

\begin{equation}\label{eq11}
 \left \{
   \begin{array}{r c l}
       \partial_t^\alpha u-\nu\Delta u+(u.\nabla)u+\nabla p & = & Cw+f\quad\text{ in } (0,T)\times\Omega \\
     \nabla.(gu) & = & 0 \quad\text{ in } (0,T)\times\Omega\\
     u(0,.)& = & u_0 \quad\text{ in }\Omega
   \end{array}
   \right .
\end{equation}
where $w:[0,T]\rightarrow U_g$ and $C:U_g\rightarrow (L_g^2(\Omega))^2$, with $U_g$ is some Hilbert space and $L_g^2(\Omega)$ is the standard Sobolev space with weight $g$.

This paper is organized as follows: In the first section, we will recall some concept and notations related to fractional calculus. Section 2, is devoted to the problem statement while section 3 will be dedicated to the proof of the existence and uniqueness of weak solutions to time-fractional g-Navier-Stokes equations. Finally section 4 will be concerned with the existence of an optimal control to \eqref{eq11}.

\section{Preliminaries}

In this section, we provide some notations and preliminary results concerning fractional calculus. For this purpose, assume $X$ to be a Banach space. Let $\alpha\in (0,1]$ and let $k_\alpha$ denote the Riemann-Liouville kernel
\[
k_\alpha(t)=\frac{t^{\alpha-1}}{\Gamma(\alpha)}
\]

For a function $v:[0,T]\rightarrow X$, we give the following definitions of derivatives and integrals:
\begin{itemize}
\item[(1)] The left Riemann-Liouville integral of $v$ is defined by
\[
_0I_t^\alpha v(t)=\int_0^t k_\alpha(t-s)v(s)\,ds,\quad t>0
\]provided the integral is point-wise defined on $[0,+\infty[$
\item[(2)] The right Riemann-Liouville integral of $v$ is defined by
\[
\, _tI_T^\alpha v(t)=\int_t^T k_\alpha(t-s)v(s)\,ds,\quad t>0
\]provided the integral is point-wise defined on $[0,+\infty[$
\item[(3)]  The left Caputo fractional derivative of order $\alpha$ of $v$, is defined by
\[
\,_0^CD_t^\alpha v(t)= \int_0^t k_{1-\alpha}(t-s)\frac{d}{ds}v(s)\,ds
\]
\item[(4)] The right Riemann-Liouville fractional derivative of order $\alpha$ of $v$ is defined by
\[
\,_tD_T^\alpha v(t)=-\frac{d}{dt} \int_t^T k_{1-\alpha}(t-s)v(s)\,ds
\]
\item[(5)] The Liouville-Weyl fractional integral on the real axis for functions $v:\mathbb R\rightarrow X$ is defined as follows
\[
\,_{-\infty}I_t^\alpha v(t)=\int_{-\infty}^t k_{\alpha}(t-s)v(s)\,ds.
\]
\item[(6)] The Caputo fractional derivative on the real axis for functions $v:\mathbb R\rightarrow X$ is defined as follows
\[
\,_{-\infty}^CD_t^\alpha v(t)= \,_{-\infty}I_t^{1-\alpha}\frac{d}{dt} v(t).
\]
\end{itemize}

Note that the notation $\partial_t^\alpha$ stands for Caputo fractional partial derivative, i.e. when functions have another argument than time. We have the following fractional integration by parts formula; see, e.g. \cite{A07}
\begin{align*}
\int_0^T(\partial_t^\alpha u(t),\psi(t))dt&=\int_0^T(u(t),_tD_T^\alpha\psi(t))dt+(u(t),_tI_T^{1-\alpha}\psi(t))|_0^T\\
&=\int_0^T(u(t),_tD_T^\alpha\psi(t))dt-(u(0),_0I_T^{1-\alpha}\psi(t))
\end{align*}since for $\psi\in C_0^\infty([0,T],X)$ one have $\displaystyle\lim_{t\to T}\,_tI_T^{1-\alpha}\psi(t)=0$.

To pass from weak convergence to strong convergence we will need a compactness result. Let $X_0$, $X$, $X_1$ be Hilbert spaces with $ X_0 \xhookrightarrow{} X \xhookrightarrow{} X_1$ being continuous and $X_0\xhookrightarrow{} X$ being compact. Assume that $v:\mathbb R\rightarrow X_1$ and denote by $\widehat v$ its Fourier transform
\[
\widehat v(\tau)=\int_{-\infty}^{+\infty}e^{-2i\pi t \tau}v(t)\,dt.
\]
We have for $\gamma>0$
\[
\widehat{_{-\infty}^CD_t^\gamma v}(\tau)=(2i\pi\tau)^\gamma \widehat v(\tau).
\]
For a given $0<\gamma<1$, we introduce the following space
\[
W^\gamma(\mathbb R,X_0,X_1)=\left\{v\in L^2(\mathbb R,X_0):\, _{-\infty}^CD_t^\gamma v\in L^2(\mathbb R,X_1)\right\}.
\]
Clearly, it is a Hilbert space for the norm
\[
\|v\|_\gamma=\left(\|v\|^2_{L^2(\mathbb R,X_0)}+\|\,|\tau|^\gamma\widehat v\|^2_{L^2(\mathbb R,X_1)}\right)^{1/2}.
\]
 For any set $K\subset \mathbb R$, we associate with it the subspace $W_K^\gamma\subset W^\gamma$ defined as
 \[
 W_K^\gamma(\mathbb R,X_0,X_1)=\{v\in W^\gamma(\mathbb R,X_0,X_1):\,\mathrm{support} u\subset K\}.
 \]

By similar discussion as in the proof of Theorem 2.2 in Temam \cite{T84}(see also Theorem 2.1 in \cite{ZP17}), it is clear that $W_K^\gamma(\mathbb R,X_0,X_1)\xhookrightarrow{} L^2(\mathbb R,X)$ is compact for any bounded set $K$ and any $\gamma>0$.

As a particular situation of the compactness result discussed above, let $H$, $V$ be two Hilbert spaces endowed with the scalar product $(.,.)_H$ and $(.,.)_V$ and the norms $|.|_H$ and $\|.\|_V$, respectively. Denote by $\langle.,.\rangle$ the dual pairing between $V$ and $V'$, the dual of $V$. Moreover assume that $V \xhookrightarrow{} H  \xhookrightarrow{} V'$ continuously and compactly and note that the space
\[
W^\gamma(0,T;V,V')=\left\{v\in L^2(0,T;V):\, \partial_t^\gamma v\in L^2(0,T;V')\right\}
\]is compactly embedded in $L^2(0,T;H)$. Similarly to Lemma 2.1 in \cite{ZP17}, we have
\[
\partial_t^\gamma(u(t),v)_V=\langle\partial_t^\gamma u(t),v\rangle
\] for $u\in W^\gamma(0,T;V,V')$ and $v\in H$. Moreover, for a differentiable function $v:[0,T]\rightarrow V$ we have from \cite{A10} that

\[
(v(t),_{~0}^{~C}D_t^\gamma v(t))_{H}\geq\frac{1}{2}\,_0^CD_t^\gamma|v(t)|^2.
\]

We end this section by the following important result
\begin{lemma}

Suppose that a nonnegative function satisfies
\[
_{~0}^{~C}D_t^\gamma v(t)+ c_1 v(t)\leq c_2(t)
\]for $c_1>0$ and $c_2$ a nonnegative integrable function for $t\in[0,T]$. Then 
\[
v(t)\leq v(0)+\frac{1}{\Gamma(\gamma)}\int_0^t(t-s)^{\gamma-1} v(s)\,ds.
\]

\end{lemma}

For more details about fractional calculus we refer to the monographs \cite{KST06,Z14,Z16}.

\section{Problem Statement }

Let $\Omega$ be a bounded domain of $\mathbb R^2$ with smooth boundary. Let $g\in W^{1,\infty}(\Omega)$ such that $0<m_0<g(x)<M_0$ for all $x\in\Omega$. To deal with the weighted divergence condition $\nabla. gu=0$ we rewrite the Navier-Stokes equations in weighted Sobolev spaces. This bring no further difficulties thanks to the definition of $g$.

Let $L_g^2(\Omega)=\left(L^2(\Omega)\right)^2$ and $H_g^1(\Omega)=(H^1(\Omega))^2$ be endowed, respectively, with the inner products
\[
(u,v)_g=\int_\Omega u.vg\,dx,\quad u,\,v\in L_g^2(\Omega)
\]
\[
((u,v))_g=\int_\Omega\sum_{j=1}^2\nabla u_j.\nabla v_j g\,dx,\quad u,\,v\in H_g^1(\Omega)
\]and norms $|.|_g^2=(.,.)_g$, $\|.\|^2_g=((.,.))_g$. Note that the norms $|.|_g$ and $\|.\|_g$ are equivalent to the usual norms in $(L^2(\Omega))^2$ and $(H^1(\Omega))^2$.

Let us introduce the following spaces
\begin{align*}
V&=\{u\in (C_0^\infty(\Omega))^2:\,\nabla.(gu)=0\,\text{ in } \Omega\},\\
H_g&=\text{ the closure of } V \text{ in } L^2_g(\Omega),\\
V_g&=\text{ the closure of } V \text{ in } H^1_g(\Omega).
\end{align*}

It follows that
\[V_g\subset H_g\simeq H_g'\subset V_g';
\] where the injections are dense and continuous. We denote by  $\|.\|_*$ the norm in $V_g'$ and by $\langle ., .\rangle$ the duality pairing between $V_g$ and its dual $V_g'$.

Set $A_g:V_g\rightarrow V_g'$ defined by $\langle A_gu,v\rangle=((u,v))_g$. Denote $D(A_g)=\{u\in V_g: A_gu\in H_g\}$. Then $D(A_g)=H^2_g(\Omega)\cap V_g$ and $A_gu=P_g\,\Delta u$ for all $u\in D(A_g)$ where $P_g$ is the orthogonal projection from $L^2_g(\Omega)$ onto $H_g$.

Set $B_g:V_g\times V_g\rightarrow V_g'$ defined by $\langle B(u,v),w\rangle=b_g(u,v,w)$, where 
\[
b_g(u,v,w)=\sum_{j,k=1}^d\int_\Omega u_j\frac{\partial v_k}{\partial x_j}w_kgdx
\]whenever the integral make sense. It is easy to check that if $u,v,w\in V_g$ then $b_g(u,v,w)=-b_g(u,w,v)$. Hence
\[
b_g(u,v,v)=0,\quad \forall u,v\in V_g.
\]
Let $u\in L^2(0,T;V_g)$, then $C_gu$ defined by
\[
(C_gu(t),v)_g=((\frac{\nabla g}{g}.\nabla)u,v)_g=b_g(\frac{\nabla g}{g},u,v),\quad \forall v\in V_g.
\]
Since
\[
-\frac{1}{g}(\nabla.g\nabla)u=-\Delta u-(\frac{\nabla g}{g}.\nabla)u
\]we have
\begin{align*}
(-\Delta u,v)_g&=((u,v))_g+((\frac{\nabla g}{g}.\nabla)u,v)_g\\
&=(A_gu,v)_g+((\frac{\nabla g}{g}.\nabla)u,v)_g\quad \forall u,v\in V_g.
\end{align*}

\begin{proposition}[\cite{R01}]\label{gstokes}
For the $g-$Stokes operator $A_g$, the following hold:
\begin{itemize}
\item[(1)] The $g-$Stokes operator $A_g$ is positive, self-adjoint with compact inverse, where the domain of $A_g$ is $D(A_g)=V_g\cap H^2_g(\Omega)$.
\item[(2)] There exists countably many eigenvalues of $A_g$ satisfying 
\begin{equation*}
0<\frac{4\pi^2m_0}{M_0}\leq \lambda_1\leq\lambda_2\leq\dots
\end{equation*}
where $\lambda_1$ is the smallest eigenvalue of $A_g$. In addition, there exists the corresponding collection of eigenfunctions $\{u_i\}_{i\in\mathbb N}$ forming an orthonormal basis for $H_g$.
\end{itemize}
\end{proposition}

Since the operators $A_g$ and $P_g$ are self-adjoint, using integration by parts, we have 
\[
(A_gu,u)_g=(P_g(-\Delta u),u)_g=\int_\Omega \nabla u.\nabla u gdx=|\nabla u|_g^2
\]
and $b_g$ satisfies
\[
|b_g(u,v,w)|_g\leq c|u|_g^{\frac{1}{2}}\|u\|_g^{\frac{1}{2}}|v|_g^{\frac{1}{2}}|w|_g^{\frac{1}{2}}\|w\|_g^{\frac{1}{2}},\quad\forall u,v,w\in V_g.
\]

Denote the operator $C_gu=P_g\left(\frac{1}{g}(\nabla g.\nabla)u\right)$ such that $(C_gu,v)_g=b_g(\frac{\nabla g}{g},u,v)$ and the operator $B_g[.]=B(.,.)$.

\begin{definition}
Let $f\in L^{2/\alpha_1}(0,T; V_g')$ for some $\alpha_1\in(0,\alpha)$ and $u_0\in H_g$ be given. A function $u\in L^2(0,T;V_g)$ with $u_t\in L^2(0,T; V_g')$ is called a weak solution of problem \eqref{eq00} if it fulfils 

\begin{equation}\label{eqdef}
 \left \{
   \begin{array}{r c l}
       \partial_t^\alpha u+\nu A_g u+\nu C_g u +B_g[u]& = & f\quad\text{ in } L^2(0,T; V_g') \\
     u(0)&=&u_0,\quad \text{  in } H_g. 
   \end{array}
   \right .
\end{equation}
\end{definition}

\section{Existence and Uniqueness}

In this section we prove the existence and uniqueness of the weak solution of problem \eqref{eqdef} under the following hypothesis:
\begin{equation*}
H(g):\quad |\nabla g|_{\infty} < \frac{1}{2}m_0\lambda_{1}^{\frac{1}{2}}
\end{equation*}where $\lambda_1$ is defined in Proposition \ref{gstokes}.
We will apply Faedo-Galerkin Method as initially used in \cite{T84} for classical Navier-Sokes equations and in \cite{Z16} for time-fractional version of it.

\begin{theorem}
Let $\Omega$ be a bounded and locally Lipschitz domain in $\mathbb R^2$. Let $f\in L^2(0,T;V_g')$ and $u_0\in H_g$. Then, under the hypothesis $H(g)$, the equation \eqref{eqdef} has  a unique weak solution $u\in L^2(0,T;V_g)\cap L^\infty(0,T; H_g)$.
\end{theorem}
\begin{proof}
We apply Faedo-Galerkin Method. Since $V_g$ is separable, there exists a sequence $\{u_i\}_{i\in\mathbb N}$ which forms a complete orthonormal system in $H_g$ and a basis for $V_g$.
Let $m$ be a positive integer. For each $m$, we define an approximate solution $u^{(m)}(t)$ of \eqref{eqdef} as a solution of the system
\begin{equation}\label{eqp1}
 \left \{
   \begin{array}{r r r}
       (_0^C\mathrm D_t^\alpha u^{(m)},u_k)_g+\nu((u^{(m)},u_k))_g+\nu b_g(\frac{\nabla g}{g},u^{(m)},u_k)+b_g(u^{(m)},u^{(m)},u_k)=(f,u_k)_g,\\
(u^{(m)}(0),u_k)_g=(u_0,u_k)_g.
   \end{array}
   \right .
\end{equation}

Let $\xi_k=\xi_k(t)$ denote the $k$th component of $u^{(m)}(t)$, i.e. $\xi_k(t)=(u^{(m)}(t),u_k)_g$. Also, let $\eta_k(t)=(f(t),u_k)_g$ be component of $f(t)$. Then \eqref{eqp1} is equivalent to a nonlinear fractional order ordinary differential equation for the functions $\xi_k$

\begin{equation}\label{eqp2}
 \left \{
   \begin{array}{r c l}
       _0^C\mathrm D_t^\alpha \xi_k+\nu\lambda_k\xi_k+\nu\displaystyle\sum_{l=1}^m b_g(\frac{\nabla g}{g},u_l,u_k)\xi_l\\
 +\displaystyle\sum_{l,l'=1}^m b_g(u_l,u_{l'},u_k)\xi_l\xi_{l'}=\eta_k&,\,k=1,\dots,m\\
\xi_k(0)=\xi_k^0.
   \end{array}
   \right .
\end{equation}

The system forms a nonlinear first order system of fractional ordinary differential equation for the functions $\xi_k(t)$ and has maximal solutions on some interval $[0,t_m]$. If $t_m<T$, then $|\xi(t)|$ must go to $+\infty$ as $t\to t_m$, where $\xi=(\xi_1,\xi_2,\dots,\xi_n)$ and $|.|$ the euclidean norm in $\mathbb R^n$. We denote $\eta=(\eta_1,\eta_2,\dots,\eta_n)$ and $(.,.)$ the usual euclidean scalar product in $\mathbb R^n$, one obtains, from \eqref{eqp2}, that 
\begin{equation*}
(_0^C\mathrm D_t^\alpha \xi(t),\xi(t))+\nu\displaystyle\sum_{k=1}^m\lambda_k\xi^2_k+\nu\displaystyle\sum_{k,l=1}^m b_g(\frac{\nabla g}{g},u_l,u_k)\xi_l\xi_k=(\eta,\xi)
\end{equation*}
because $\displaystyle\sum_{k,l,l'=1}^m b_g(u_l,u_{l'},u_k)\xi_l\xi_{l'}\xi_k=0$. It then follows that

\begin{align*}
\frac{1}{2}\,_0^C\mathrm D_t^\alpha |\xi(t)|^2+\nu\displaystyle\sum_{k=1}^m\lambda_k\xi^2_k&\leq(\eta,\xi)-\nu\displaystyle\sum_{k,l=1}^m b_g(\frac{\nabla g}{g},u_l,u_k)\xi_l\xi_k\\
&\leq \left(\displaystyle\sum_{k=1}^m \lambda_k^{-1}\eta_k^2\right)^{\frac{1}{2}}\left(\displaystyle\sum_{k=1}^m\lambda_k\xi_k^2\right)^{\frac{1}{2}}+\frac{\nu|\nabla g|_\infty}{m_0}\displaystyle\sum_{k=1}^m|u_k||\nabla u_k|\\
&\leq \frac{1}{2\lambda_1\nu}|\eta(t)|^2+\frac{\nu}{2}\displaystyle\sum_{k=1}^m\lambda_k\xi^2_k+\frac{\nu|\nabla g|_\infty}{m_0\lambda_1^{\frac{1}{2}}}\displaystyle\sum_{k=1}^m\lambda_k\xi^2_k.
\end{align*}
It follows that 
\begin{equation*}
_0^C\mathrm D_t^\alpha |\xi(t)|^2+\nu'\displaystyle\sum_{k=1}^m\lambda_k\xi^2_k\leq \frac{1}{\lambda_1\nu}|\eta(t)|^2
\end{equation*}where $\nu'=\nu\left(1-\frac{2|\nabla g|_\infty}{m_0\lambda_1^{\frac{1}{2}}}\right)>0$ (because of $H(g)$). It follows that

\begin{equation*}
_0^C\mathrm D_t^\alpha |\xi(t)|^2+\nu'\lambda_1|\xi(t)|^2\leq \frac{1}{\lambda_1\nu}|\eta(t)|^2.
\end{equation*}
We have the estimates

\begin{align*}
|\xi(t)|^2 \leq&\, |\xi(0)|^2+\frac{1}{\nu'\lambda_1\Gamma(\alpha)}\int_0^t(t-s)^{\alpha-1}|\eta(s)|^2\,ds\\
\leq &\,|\xi(0)|^2+\frac{1}{\nu'\lambda_1\Gamma(\alpha)}\int_0^t|\eta(s)|^{2/\alpha_1}\,ds+\frac{1}{\nu'\lambda_1\Gamma(\alpha)}\int_0^t(t-s)^{\frac{\alpha-1}{1-\alpha_1}}\,ds\\
\leq&\, |\xi(0)|^2+\frac{1}{\nu'\lambda_1\Gamma(\alpha)}\int_0^t|\eta(s)|^{2/\alpha_1}\,ds+\frac{T^{1+b}}{(1+b)\nu'\lambda_1\Gamma(\alpha)}\\
\leq &\, M
\end{align*}
where $\alpha_1\in(0,\alpha)$, $b=\frac{\alpha-\alpha_1}{1-\alpha_1}$. Therefore $t_m=T$.

We multiply \eqref{eqp1} by $\xi_k(t)$ and sum this equations for $k=1,\dots,m$. Taking into account that $b_g(u^{(m)},u^{(m)},u^{(m)})=0$, we get

\begin{equation*} 
       (_0^C\mathrm D_t^\alpha u^{(m)},u^{(m)})_g+\nu\|u^{(m)}(t)\|^2_g+\nu b_g(\frac{\nabla g}{g},u^{(m)}(t),u^{(m)}(t))=(f,u^{(m)}(t))_g .
\end{equation*}
Using Schwartz and Young inequalities, we get
\begin{align*}
\frac{1}{2}\,_0^C\mathrm D_t^\alpha |u^{(m)}(t)|^2+\nu\|u^{(m)}(t)\|_g^2\leq&\, \frac{1}{2\nu}\|f(t)\|_{V_g'}^2+\frac{\nu}{2}\|u^{(m)}(t)\|^2_g+\frac{\nu|\nabla g|^2_\infty}{m_0^2}|u^{(m)}(t)|_g^2\\
\leq &\, \frac{1}{2\nu}\|f(t)\|_{V_g'}^2+\frac{\nu}{2}\|u^{(m)}(t)\|^2_g+\frac{\nu|\nabla g|^2_\infty}{\lambda_1m_0^2}\|u^{(m)}(t)\|_g^2.
\end{align*}
It follows that
\begin{equation*}
_0^C\mathrm D_t^\alpha |u^{(m)}(t)|_g^2+\nu'\|u^{(m)}(t)\|_g^2\leq\, \frac{1}{\nu}\|f(t)\|_{V_g'}^2
\end{equation*}where $\nu'=\nu\left(1-\frac{2|\nabla g|^2_\infty}{\lambda_1m_0^2}\right)$. Integrating (with order $\alpha$) we obtain

\begin{align*}
|u^{(m)}(t)|_g^2 +\nu'\int_0^t (t-s)^{\alpha-1}\|u^{(m)}(s)\|_g^2\, ds\leq&\, |u_{0m}|_g^2+\frac{1}{\nu'}\int_0^t(t-s)^{\alpha-1}\| f(s)\|_{V_g'}^2\,ds\\
\leq &\,|u_{0m}|_g^2+\frac{1}{\nu'}\int_0^t\|f(s)\|_{V_g'}^{2/\alpha_1}\,ds+\frac{1}{\nu'}\int_0^t(t-s)^{\frac{\alpha-1}{1-\alpha_1}}\,ds\\
\leq&\, |u_{0m}|_g^2+\frac{1}{\nu'}\int_0^t\|f(s)\|_{V_g'}^{2/\alpha_1}\,ds+\frac{T^{1+b}}{(1+b)\nu'}.
\end{align*}
It follows that for almost every $t\in[0,T]$
\begin{equation}\label{eqp3}
\displaystyle\sup_{t\in[0,T]}|u^{(m)}(t)|_g^2\leq |u_{0m}|_g^2+\frac{1}{\nu'}\int_0^T\|f(s)\|_{V_g'}^{2/\alpha_1}\,ds+\frac{T^{1+b}}{(1+b)\nu'}
\end{equation}
\begin{equation}\label{eqp4}
\int_0^t (t-s)^{\alpha-1}\|u^{(m)}(s)\|_g^2\, ds\leq \frac{1}{\nu'}|u_{0m}|_g^2+\frac{1}{\nu'^2}\int_0^t\|f(s)\|_{V_g'}^{2/\alpha_1}\,ds+\frac{T^{1+b}}{(1+b)\nu'^2}.
\end{equation}
The relation \eqref{eqp3} ensures that the sequence $\{u^{(m)}\}_m$ is bounded in $L^\infty(0,T;H_g)$. Moreover, from \eqref{eqp4}, one gets
\begin{align*}
T^{\alpha-1} \int_0^t\|u^{(m)}(s)\|_g^2\, ds\leq&\, \int_0^t (t-s)^{\alpha-1}\|u^{(m)}(s)\|_g^2\, ds\\
\leq&\,\frac{1}{\nu'}|u_{0m}|_g^2+\frac{1}{\nu'^2}\int_0^T\|f(s)\|_{V_g'}^{2/\alpha_1}\,ds+\frac{T^{1+b}}{(1+b)\nu'^2}.
\end{align*}
Hence the sequence $\{u^{(m)}\}_m$ is bounded in $L^2(0,T;V_g)$.

Let $\tilde u^{(m)}:\mathbb R\rightarrow V_g$ denote the function defined by
\begin{equation*}
 \tilde u^{(m)}(t)=
 \left \{
   \begin{array}{r c l}
    u^{(m)}(t),&\quad 0\leq t\leq T   \\
    0,&\quad \text{ otherwise}
   \end{array}
   \right .
\end{equation*}
and $\hat u^{(m)}$ denotes the Fourier transform of $\tilde u^{(m)}$. We show that the sequence $\{\tilde u^{(m)}\}_m$ remains bounded in $W^\gamma(\mathbb R, V_g, H_g)$. To do so, we need to verify that 
\begin{equation}\label{eqgamma}
\int_{-\infty}^{+\infty}|\tau|^{2\gamma}|\hat u^{(m)}(\tau)|^2\,d\tau\leq \mathrm{const.}\quad\text{for some }\gamma>0.
\end{equation}

In order to prove \eqref{eqgamma}, we observe that 
\begin{equation}\label{Fm}
(_0^CD_t^\alpha\tilde u^{(m)},u_k)_g=(\widetilde F_m,u_k)_g+(u_{m0},u_k)_g\,_{-\infty}I_t^{1-\alpha}\delta_0-( u^{(m)}(T),u_k)_g \,_{-\infty}I_t^{1-\alpha}\delta_T
\end{equation}
where $F_m=f-\nu A_g u^{(m)}-B_g(u^{(m)},u^{(m)})-\nu C_g u^{(m)}$ and $\delta_0,\,\delta_T$ are Dirac distributions at $0$ and $T$. Here $\widetilde F_m$ is defined as usual by

\begin{equation*}
 \widetilde F_m(t)=
 \left \{
   \begin{array}{r c l}
    F_m(t),&\quad 0\leq t\leq T   \\
    0,&\quad \text{ otherwise}.
   \end{array}
   \right .
\end{equation*}

Indeed, it is classical that since $\tilde u^{(m)}$ has two discontinuities at $0$ and $T$, the Caputo derivative of $\tilde u^{(m)}$ is given by

\begin{align*}
_{-\infty}^CD_t^\alpha\tilde u^{(m)}&= \,_{-\infty}I_t^{1-\alpha}\left(\frac{d}{dt}\tilde u^{(m)}\right)\\
&=\,_{-\infty}I_t^{1-\alpha}\left(\frac{d}{dt} u^{(m)}+u^{(m)}(0)\delta_0-u^{(m)}(T)\delta_T\right)\\
&=\,_{0}^CD_t^\alpha u^{(m)}+_{-\infty}I_t^{1-\alpha}\left(u^{(m)}(0)\delta_0-u^{(m)}(T)\delta_T\right).
\end{align*}
By the Fourier transform \eqref{Fm} yields to
\begin{align*}
(2i\pi\tau)^\alpha(\hat u^{(m)},u_k)_g=&(\widehat{F}_m,u_k)_g+(u_{m0},u_k)_g(2i\pi\tau)^{\alpha-1}\\
&\qquad\qquad-(u^{(m)}(T),u_k)_g(2i\pi\tau)^{\alpha-1}e^{-2i\pi T\tau}
\end{align*}here $\hat u^{(m)}$ and $\widehat{F}_m$ denote the Fourier transforms of $\tilde u^{(m)}$ and $\widetilde{F}_m$, respectively.

We multiply by $\hat \xi_k(\tau)$ and sum these equations for $k=1,\dots,m$ to get 
\begin{align*}
(2i\pi\tau)^\alpha|\hat u^{(m)}(\tau)|_g^2=&(\widehat{F}_m(\tau),\hat u^{(m)}(\tau))_g+(u_{m0},\hat u^{(m)}(\tau))_g(2i\pi\tau)^{\alpha-1}\\
&\qquad\qquad-(u^{(m)}(T),\hat u^{(m)}(\tau))_g(2i\pi\tau)^{\alpha-1}e^{-2i\pi T\tau}.
\end{align*}
We have
\begin{align*}
\int_0^T\|F_m(t)\|_{V_g'}\,dt\leq c\,\int_0^T(\|f(t)\|_{V_g'}+&|u^{(m)}(t)|_g\|u^{(m)}(t)\|_g+\|u^{(m)}(t)\|_g\\
&+|\nabla g|_\infty \|u^{(m)}(t)\|_g)dt.
\end{align*}
Therefore $\|F_m(t)\|_{V_g'}$ is bounded in $L^1(0,T;V_g')$. Hence 
\[
\displaystyle\sup_{\tau\in\mathbb R}\|\widehat F_m(\tau)\|_{V_g'}\leq c,\quad \forall m.
\]
Moreover, since $u_{m0}$ and $u^{(m)}(T)$ are bounded we have
\begin{align*}
|\tau|^\alpha|\hat u^{(m)}(\tau)|^2_g&\leq c_2\|\hat u^{(m)}(\tau)\|_g+c_3|\tau|^{\alpha-1}|\hat u^{(m)}(\tau)|_g\\
&\leq c_4(1\vee|\tau|^{\alpha-1})\|\hat u^{(m)}(\tau)\|_g.
\end{align*}
For $\gamma$ fixed, $\gamma<\alpha/4$, we observe that 
\[
|\tau|^{2\gamma}\leq c(\gamma)\frac{1+|\tau|^\alpha}{1+|\tau|^{\alpha-2\gamma}}.
\]
Then we can write
\begin{align*}
\int_{-\infty}^{+\infty}|\tau|^{2\gamma}|\hat u^{(m)}(\tau)|^2_g\leq & c_5(\gamma) \int_{-\infty}^{+\infty}\frac{1+|\tau|^\alpha}{1+|\tau|^{\alpha-2\gamma}}|\hat u^{(m)}(\tau)|^2_g\,d\tau\\
\leq & c_6(\gamma) \int_{-\infty}^{+\infty}\frac{1}{1+|\tau|^{\alpha-2\gamma}}\|\hat u^{(m)}(\tau)\|^2_g\,d\tau\\
&\qquad\qquad+c_7(\gamma) \int_{-\infty}^{+\infty}\frac{|\tau|^{\alpha-1}}{1+|\tau|^{\alpha-2\gamma}}\|\hat u^{(m)}(\tau)\|^2_g\,d\tau.
\end{align*}
By Parseval inequality, the first integral is bounded as $m\to\infty$. Applying the Schwartz inequality, the second integral yields to
\begin{align*}
\int_{-\infty}^{+\infty}\frac{|\tau|^{\alpha-1}}{1+|\tau|^{\alpha-2\gamma}}\|\hat u^{(m)}(\tau)\|^2_g\,d\tau\leq&\left(\int_{-\infty}^{+\infty}\frac{d\tau}{(1+|\tau|^{\alpha-2\gamma})^2}\right)^{1/2}\\
&\qquad\qquad\times\left(\int_{-\infty}^{+\infty}|\tau|^{2\alpha-2}\|\hat u^{(m)}(\tau)\|^2_g\,d\tau\right)^{1/2}.
\end{align*}

The first integral is finite due to $\gamma<\alpha/4$. On the other hand, it follows from the Parseval equality that
\begin{align*}
\int_{-\infty}^{+\infty}|\tau|^{2\alpha-2}\|\hat u^{(m)}(\tau)\|^2_g\,d\tau&=\int_{-\infty}^{+\infty}\|\,_{-\infty}\mathrm I_t^{1-\alpha}\tilde u^{(m)}(t)\|_g^2\,dt\\
&=\int_0^T\|\,_0\mathrm I_t^{1-\alpha} u^{(m)}(t)\|^2_g\,dt\\
&\leq\left(\frac{T^{1-\alpha}}{\Gamma(2-\alpha)}\right)^2\int_0^T\|u^{(m)}(t)\|_V^2\,dt.
\end{align*}

Which implies that \eqref{eqgamma} holds.

We know that a subsequence of $\{u^{(m)}\}_m$ (which we will denote with the same symbol) converges to some $u$ weakly in $L^2(0,T;V_g)$ and weak-star in $L^\infty(0,T;H_g)$ with $u\in L^2(0,T;V_g)\cap L^\infty(0,T;H_g)$. As $W^\gamma(0,T,V_g;H_g)$ is compactly embedded in $L^2(0,T; H_g)$ then $\{u^{(m)}\}_m$ strongly converges in $L^2(0,T;H_g)$.

In order to pass to the limit, we consider the scalar function $\psi$ continuously differentiable  on $[0,T]$ and such that $\psi(T)=0$. In the first equation of the system \eqref{eqp1}, we consider $\psi(t)u_k$ instead of $u_k$ and then we integrate by parts. This leads to the equation
\begin{align*}
\int_0^T&(u^{(m)}(t),_t\mathrm D_T^\alpha\psi(t)u_k)_g\,dt+\int_0^T b_g(u^{(m)}(t),u^{(m)}(t),\psi u_k)\,dt+\nu\int_0^T((u^{(m)}(t),\psi u_k))_g\\
&+\nu\int_0^Tb_g(\frac{\nabla g}{g},u^{(m)}(t),\psi u_k)\,dt=(u_{0m},_0\mathrm I_T^{1-\alpha}\psi(t) u_k)_g+\int_0^T(f(t),u_k)_g\,dt.
\end{align*}
Moreover we have the following convergence
\begin{align*}
\displaystyle\lim_{m\to+\infty} \int_0^T b_g(u^{(m)}(t),u^{(m)}(t),&\psi(t) u_k)\,dt=-\displaystyle\lim_{m\to+\infty} \int_0^T b_g(u^{(m)}(t),\psi u_k,u^{(m)}(t))\,dt\\
&=-\displaystyle\lim_{m\to+\infty} \displaystyle\sum_{i,j=1}^d\int_0^T\int_\Omega u^{(m)}_i \mathrm D_i(u_k)_j u^{(m)}_j\,dx\,\psi(t)\,gdxdt\\
 &=- \displaystyle\sum_{i,j=1}^d\int_0^T\int_\Omega (u)_i \mathrm D_i(u_k)_j (u)_j\,dx\,\psi(t)\,gdxdt\\
&=- \int_0^T b_g(u(t),\psi u_k,u(t))\,dt \\
&=\int_0^T b_g(u(t),u,\psi(t) u_k)\,dt.
\end{align*}where $D_i$ stands for the partial derivative with respect to $x_i$.
In the same way, one can prove that 

\[
\displaystyle\lim_{m\to+\infty}\int_0^Tb_g(\frac{\nabla g}{g},u^{(m)}(t),\psi u_k)\,dt=\int_0^Tb_g(\frac{\nabla g}{g},u(t),\psi u_k)\,dt.
\]

It then follows that 
\begin{align}\label{eqproof}
\int_0^T&(u(t),_t\mathrm D_T^\alpha\psi(t)u_k)_g\,dt+\int_0^T b_g(u(t),u(t),\psi u_k)\,dt+\nu\int_0^T((u(t),\psi u_k))_g\\
&+\nu\int_0^Tb_g(\frac{\nabla g}{g},u(t),\psi u_k)\,dt=(u_{0},\,_0\mathrm I_T^{1-\alpha}\psi(t) u_k)_g+\int_0^T(f(t),u_k)_g\,dt.
\end{align}
This equation holds for $v$ which is finite linear combination of $u_k$, $k=1,\dots,m$ and by continuity it holds for any $v$ in $V_g$. It then follows that $u$ satisfies the equation \eqref{eqdef}. To end the proof it still to check that $u$ satisfies the initial condition $u(0)=u_0$. To do so it suffices to multiply \eqref{eqdef} by $\psi$ and integrate. By making use of the integration by parts and comparing with \eqref{eqproof}, one can find that 
\[
(u_{0}-u(0), v)_g\,_0\mathrm I_T^{1-\alpha}\psi(t)=0
\]which lead to the desired result by taking a particular choice of $\psi$.
\end{proof}
\begin{theorem}
The solution u of problem \eqref{eqdef} is unique.
\end{theorem} 
\begin{proof}
Let $u_1$ and $u_2$ be two weak solutions with the same initial condition. Let $w=u_1-u_2$. Then we have
\[
(\,_0^CD_t^\alpha w,v)_g+b_g(u_1,u_2,v)-b_g(u_2,u_2,v)+\nu((w,v))_g+\nu b_g(\frac{\nabla g}{g},w,v)=0.
\]
Taking $v=w$, one obtains
\begin{align*}
\,_0^CD_t^\alpha |w|_g^2+2\nu\|w\|_g^2&\leq 2b_g(w,w,u_2)-2\nu b_g(\frac{\nabla g}{g},w,w)\\
&\leq c_0|w|_g\|w\|_g\|u_2\|_g+2\nu\frac{|\nabla g|_\infty}{m_0\lambda^{1/2}}\|w\|_g^2\\
&\leq \nu\|w\|_g^2+c_1|w|^2_g\|u_2\|^2_g+2\nu\frac{|\nabla g|_\infty}{m_0\lambda^{1/2}}\|w\|_g^2.
\end{align*}
It follows 

\begin{equation*}
\,_0^CD_t^\alpha |w|_g^2+\nu'\|w\|_g^2\leq c_1|w|^2_g\|u_2\|^2_g
\end{equation*}
where $\nu'=\nu(1-\frac{2|\nabla g|_\infty}{m_0\lambda^{1/2}})>0$. Hence
\begin{equation*}
\,_0^CD_t^\alpha |w|_g^2\leq c_1|w|^2_g\|u_2\|^2_g.
\end{equation*}
It follows
\[ 
|w(t)|_g^2\leq |w(0)|^2_g+\frac{c_2}{\Gamma(\alpha)}\int_0^t(t-s)^{\alpha-1}|w(s)|_g^2\|u_2(s)\|_g^2\,ds.
\]
and by Gronwall inequality
\begin{align*} 
|w(t)|_g^2&\leq |w(0)|^2_g\,\exp\left(\frac{c_2}{\Gamma(\alpha)}\int_0^t(t-s)^{\alpha-1}\|u_2(s)\|_g^2\,ds\right)\\
&\leq  |w(0)|^2_g\,\exp\left(\frac{c_2}{\Gamma(\alpha)}\left(|u_2(0)|_g^2+\frac{1}{\nu'}\int_0^T\|f(s)\|_{V_g'}^{2/\alpha_1}+\frac{T^{1+b}}{(1+b)\nu'}\right)\right).
\end{align*}
Since $|w(0)|_g=0$, it follows that $|w|_g=0$, which complete the proof.
\end{proof}

\section{Optimal Control}

Let $\mathcal U_g=L^{2/\alpha_1}(0,T;H_g)$ be the space of controls. For every $f\in\mathcal U_g$, we denote by $ S(f)\subset L^2(0,T;V_g)\cap L^\infty(0,T;H_g):=\mathcal V_g$ the solution set corresponding to $f$ of the problem \eqref{eqdef}. It is then clear, by definition, that $ S(f)$ is nonempty for all $f\in\mathcal U_g$.

\begin{theorem}\label{fn}
Assume that $f_n,\,f\in L^{2/\alpha_1}(0,T;H_g)$ such that $f_n$ converges to $f$ weakly in $L^{2/\alpha_1}(0,T;H_g)$. Then, for every sequence $\{u_n\}_n$ such that $u_n\in S(f_n)$, we can find a subsequence of $\{u_n\}_n$ that converges weakly in $L^2(0,T,V_g)$ and weakly-star in $L^\infty(0,T;H_g)$ to some $u\in L^2(0,T;V_g)\cap L^\infty(0,T;H_g)$ such that $u\in S(f)$.
\end{theorem}

\begin{proof}
Let $f_n,\,f\in L^{2/\alpha_1}(0,T;H_g)$ such that
\[
f_n\rightharpoonup f \text{ weakly in }L^{2/\alpha_1}(0,T;H_g).
\]
There exists $u_n\in L^2(0,T;V_g)\cap L^\infty(0,T;H_g)$ such that
\begin{align*}
\,_0^C D_t^\alpha u_n(t)+\nu A_g &u_n(t)+\nu C_g u_n(t)\\&+B_g(u_n(t),u_n(t))=f_n(t)\quad \text{ for a.e. }t\in[0,T].
\end{align*}
Multiplying by $u_n$, one obtains
\[
\, _0^C D_t^\alpha |u_n(t)|_g^2+\nu'\|u_n(t)\|_g^2\leq \frac{1}{\nu}\|f_n(t)\|_{g'}^2.
\]
This yields to 
\begin{align*}
|u(t)|_g^2+\nu'\int_0^t(t-s)^{\alpha-1}\|u(s)\|^2_g\,ds&\leq|u(0)|_g^2+\frac{1}{\nu}\int_0^t(t-s)^{\alpha-1}\|f_n(s)\|^2_{g'}\,ds\\
&\leq |u(0)|_g^2+\frac{1}{\nu}\int_0^T\|f_n(s)\|^{\frac{2}{\alpha_1}}_{g'}\,ds+\frac{T^{1+b}}{(1+b)\nu}.
\end{align*}
Therefore the sequence $\{u_n\}_n$ is bounded in $L^2(0,T;V_g)$ and in $L^\infty(0,T;H_g)$. Thus there exists an element $\hat u\in L^2(0,T;V_g)\cap L^\infty(0,T;H_g)$ such that $\{u_n\}_n$ converges weakly in $L^2(0,T;V_g)$ and star-weakly in $L^\infty(0,T;H_g)$. Moreover, with the same argument as in the existence theorem,  the sequence $\{u_n\}_n$ converges to $\hat u$ strongly in  $L^2(0,T;H_g)$. Finally, from the continuity of $A_g$ and $C_g$ and the weak continuity of $B_g$ we deduce that $\hat u\in S(f)$.

\end{proof}

\begin{corollary}\label{corn1}
Assume that $\varphi_n,\,\varphi,\,h\in L^{2/\alpha_1}(0,T;H_g)$ such that $\varphi_n$ converges to $\varphi$ weakly in $L^{2/\alpha_1}(0,T;H_g)$. Then, for every sequence $\{u_n\}_n$ such that $u_n\in S(\varphi_n+h)$, we can find a subsequence of $\{u_n\}_n$ that converges weakly in $L^2(0,T,V_g)$ and weakly-star in $L^\infty(0,T;H_g)$ to some $u\in L^2(0,T;V_g)\cap L^\infty(0,T;H_g)$ such that $u\in S(\varphi+h)$.
\end{corollary}
\begin{proof}
It suffices to take $f_n=\varphi_n+f$ and apply Theorem \ref{fn}.
\end{proof}

Let $\mathcal U_g$ be a real Hilbert space and $C\in\mathscr L(\mathcal U_g,H_g)$ a bounded linear operator from $\mathcal U_g$ to $H_g$. We have the following corollary:

\begin{corollary}\label{corn2}
Assume that $w_n,\,w\in L^{2/\alpha_1}(0,T;\mathcal U_g)$ and $f\in L^{2/\alpha_1}(0,T;V'_g)$ such that $w_n$ converges to $w$ weakly in $L^{2/\alpha_1}(0,T;\mathcal U_g)$. Then, for every sequence $\{u_n\}_n$ such that $u_n\in S(Cw_n+f)$, we can find a subsequence of $\{u_n\}_n$ that converges weakly in $L^2(0,T,V_g)$ and weakly-star in $L^\infty(0,T;H_g)$ to some $u\in L^2(0,T;V_g)\cap L^\infty(0,T;H_g)$ such that $u\in S(Cw+f)$.
\end{corollary}
\begin{proof}
It suffices to take $f_n=Cw_n$ and apply Corollary \ref{corn1}.
\end{proof}
Let $\mathcal U_{ad}$ be a nonempty subset of $\mathcal U_g$ consisting of admissible controls. Let $\mathscr F:\mathcal U_g\times\mathcal V_g\rightarrow \mathbb R$ be the objective functional we want to minimize. The control problem reads as follows: Find a control $\hat f\in\mathcal U_{ad}$ and a state $\hat u\in S(\hat f)$ such that
\begin{equation}\label{control}
\mathscr F(\hat f,\hat u)=\inf\left\{\mathscr F(f,u):\,f\in\mathcal U_{ad},\, u\in S( f)  \right\}.
\end{equation}
A couple which solves \eqref{control} is called an optimal solution. The existence of such optimal control can be proved by using Theorem \ref{fn}. To do so, we need the following additional hypotheses:

\begin{itemize}
\item[$H(\mathcal U_{ad})$]$\quad \mathcal U_{ad}$ is a bounded and weakly closed subset of $\mathcal U$.
\end{itemize}

\begin{itemize}
\item[$H(\mathscr F)$] $\quad \mathscr F$ is lower semicontinuous with respect to $\mathcal U\times\mathcal V$ endowed with the weak topology.
\end{itemize}

\begin{theorem}\label{existencecontrol}
Assume that $H(g)$, $H(\mathcal U_{ad})$ and $H(\mathscr F)$ are fulfilled. Then the problem \eqref{control} has an optimal control.
\end{theorem}

\begin{proof}
Let $(f_n,u_n)$ be a minimizing sequence for the problem \eqref{control}, i.e $f_n\in\mathcal U_{ad}$ and $u_n\in S(f_n)$ such that 
\[
\displaystyle\lim_{n\to\infty}\mathscr F(f_n,u_n)=\inf\left\{\mathscr F(f,u):\,f\in\mathcal U_{ad},\, u\in S( f) \right\}=:m.
\]
It follows that the sequence $f_n$ belongs to a bounded subset of the reflexive Banach space $L^{2/\alpha_1}(0,T; V_g)$. We may then assume that $f_n\rightarrow \hat f$ weakly in $L^{2/\alpha_1}(0,T;V_g')$ (by passing to a subsequence if necessary). By $H(\mathcal U_{ad})$, we have $\hat f\in\mathcal U_{ad}$. From Theorem \ref{fn}, we obtain, by again passing to a subsequence if necessary, that $u_n\rightarrow \hat u$ weakly in $L^2(0,T;V_g)$ and star-weakly in $L^\infty(0,T;H_g)$ with $\hat u\in S(\hat f)$. By $H(\mathscr F)$, we have $m\leq\mathscr F(\hat f,\hat u)\leq \displaystyle\liminf_{n\to\infty}\mathscr F(f_n,u_n)=m$. Which completes the proof.
\end{proof}

Let $f\in L^{2/\alpha_1}(0,T;V_g')$ and $z\in L^{\frac{2}{\alpha_1}}(0,T;H_g)$. Consider the following optimal problem $(P)$:
\[
\mathrm{Minimize}\, J(u,w)=\frac{1}{2}\int_0^T\int_\Omega(u(t,x)-u(s,x))^2\,g(x)dx\,dt+\int_0^Th(w(t))\,dt
\]over $(u,w)\in \left(L^2(0,T;V_g)\cap L^\infty(0,T;H_g)\right)\times L^{2/\alpha_1}(0,T;U_g)$ subject to $u\in S(Cw+f)$, i.e. $u$ satisfies

\begin{equation*}
 \left \{
   \begin{array}{r c l}
\,_0^C D_t^\alpha u+\nu A_g u+\nu C_g u+B_g(u,u)=&Cw+f\\
     u(0)=&u_0 .
   \end{array}
   \right .
\end{equation*}

We assume that:

\begin{itemize}
\item[$H(h)\quad$] The function $h:U_g\rightarrow \mathbb R$ is convex, lower semicontinuous and satisfies 
\[
|h(w)|\geq b_1|w|_{U_g}^{\frac{2}{\alpha_1}}+b_2
\]for some $b_1>0,\,b_2\in\mathbb R$.
\end{itemize}
\begin{theorem}
Assume $H(g)$ and $H(h)$. Then problem $(P)$ has at least one solution $(\hat u,\hat w)\in \left(L^2(0,T;V_g)\cap L^\infty(0,T;H_g)\right)\times L^{2/\alpha_1}(0,T;U_g)$.
\end{theorem}

\begin{proof}
Let $(u_n,w_n)$ be a minimizing sequence of $(P)$. By $H(h)$, $(w_n)_n$ is bounded in $L^{2/\alpha_1}(0,T;U_g)$. Hence, there exists a subsequence converging weakly in $L^{2/\alpha_1}(0,T;U_g)$ to some $w$. It follows by Corollary \ref{corn2} that $u_n$ converges weakly in $L^2(0,T;V_g)$ and weakly-star in $L^\infty(0,T;H_g)$ to some $u\in S(Cw+f)$. Take $\mathcal U_{ad}=L^{2/\alpha_1}(0,T;\mathrm{Im}\,C)$ and note that $J(u,w)=\mathscr F(Cw+f,u)$. It follows by Theorem \ref{existencecontrol}, that problem $(P)$ has at least one solution as required.
\end{proof}

\par\bigskip

\subsection*{Acknowledgments}
We would like to express our gratitude to the Editor for taking time to handle the manuscript and to anonymous referees whose constructive comments are very helpful for improving the quality of our paper.

\end{document}